\documentclass[12pt,a4paper]{article}
\usepackage{latexsym,amssymb,amsmath,amsthm}
\usepackage{enumitem}
\usepackage{tikz}
\tikzstyle{vertex}=[circle, fill=black, minimum size=5pt, inner sep=0]

\newtheorem{theorem}{Theorem}[section]
\newtheorem{lemma}[theorem]{Lemma}

\newtheorem{proposition}[theorem]{Proposition}
\newtheorem{observation}[theorem]{Observation}
\newtheorem{problem}[theorem]{Problem}
\newtheorem{conjecture}[theorem]{Conjecture}
\newtheorem{claim}{Claim}

\newcommand{\claimproofend}{\hspace*{.1mm}\hspace{\fill}}

\newcommand\Setx[1] {\left\{{#1}\right\}}

\DeclareMathOperator{\RP}{\mathbb RP}
\DeclareMathOperator{\RR}{\mathbb R}

\begin{document}
\title{\textbf{Criticality in Sperner's Lemma}} \author{Tom\'{a}\v{s}
  Kaiser$^{\:1}$\and Mat\v{e}j Stehl\'{\i}k$^{\:2}$ \and Riste
  \v{S}krekovski$^{\:3}$}

\date{}

\maketitle
\begin{flushright}
  \emph{To the memory of Frank H. Lutz (1968--2023)\qquad}
\end{flushright}\vspace{4mm}

\begin{abstract}
  We answer a question posed by T.~Gallai in 1969 concerning criticality in
  Sperner's lemma, listed as Problem~9.14 in the collection of Jensen and Toft
  [Graph coloring problems, John Wiley \& Sons, Inc., New York, 1995].
  
  Sperner's lemma states that if a labelling of the vertices of a triangulation
  of the $d$-simplex $\Delta^d$ with labels $1, 2, \ldots, d+1$ has the
  property that (i) each vertex of $\Delta^d$ receives a distinct label, and
  (ii) any vertex lying in a face of $\Delta^d$ has the same label as one of
  the vertices of that face, then there exists a rainbow facet (a facet whose
  vertices have pairwise distinct labels). For $d\leq 2$, it is not difficult
  to show that for every facet $\sigma$, there exists a labelling with the
  above properties where $\sigma$ is the unique rainbow facet. For every $d\geq
  3$, however, we construct an infinite family of examples where this is not
  the case, which implies the answer to Gallai's question as a corollary. The
  construction is based on the properties of a $4$-polytope which had been used
  earlier to disprove a claim of T.~S.~Motzkin on neighbourly polytopes.
\end{abstract}
\footnotetext[1]{Department of Mathematics and European Centre of Excellence
  NTIS (New Technologies for the Information Society), University of West
  Bohemia, Pilsen, Czech Republic. E-mail: \texttt{kaisert@kma.zcu.cz}. Supported
  by project GA20-09525S of the Czech Science Foundation.}%
\footnotetext[2]{Université Paris Cité, CNRS, IRIF, F-75006, Paris, France.
  E-mail: \texttt{matej@irif.fr}.  Partially supported by ANR project DISTANCIA
  (ANR-17-CE40-0015).}%
\footnotetext[3]{Faculty of Mathematics and Physics, University of Ljubljana,
  1000 Ljubljana \& Faculty of Information Studies, 8000 Novo Mesto, Slovenia.
  E-mail: \texttt{skrekovski@gmail.com}.  Partially supported by  the Slovenian
  Research Agency Program P1-0383, Project J1-3002, and
  BI-FR/22-23-Proteus-011.}%

\section{Introduction}

A central result of combinatorial topology, Sperner's lemma~\cite{Spe28} has
found wide-ranging applications in diverse areas of mathematics and beyond; for
example, it can be used to give an elementary proof of Brouwer's fixed point
theorem~\cite[p.~117]{Mun84} from algebraic topology, and to solve various
fair-division problems, such as rental harmony~\cite{Su99,Sun14}.

Let $\Delta^d$ be a $d$-simplex with vertices $v_1, \ldots, v_{d+1}$, and let
$K$ be a finite triangulation of $\Delta^d$ (see Section~\ref{sec:terminology}
for a precise definition). A \emph{Sperner labelling} of $K$ is a map $\lambda:
V(K) \to \{1,\ldots, d+1\}$ such that $\lambda(v_i)=i$, and every vertex lying
in a face of $\Delta^d$ has the same label as one of the vertices of that face.
A simplex is said to be \emph{rainbow} if its vertices receive pairwise
distinct labels. The following classical result, proved by
Sperner~\cite{Spe28}, is usually referred to as \emph{Sperner's lemma} (see
e.g.~\cite[p.~423]{DLGMM19}, \cite[p.~3]{dL13}).

\begin{theorem}[Sperner's lemma]\label{thm:sperner}
  If $K$ is a finite triangulation of the $d$-simplex $\Delta^d$ and
  $\lambda:V(K) \to \{1, \ldots, d+1\}$ is any Sperner labelling of $K$, then
  there is a rainbow facet.
\end{theorem}

A natural question asked by Gallai in 1969 (see Problem~9.14 in Jensen and
Toft~\cite{JT95} and Problem~\ref{prob:gallai} below) concerns `criticality'
in Sperner's lemma. His question contains the following subproblem.

%

\begin{problem}\label{prob:unique}
  Given any triangulation $K$ of $\Delta^d$, and any facet $\sigma \in K$, is
  there a Sperner labelling of $K$ where $\sigma$ is the unique rainbow facet?
\end{problem}

We will see later that if one could answer Problem~\ref{prob:gallai} in the
affirmative, then the same would follow for Problem~\ref{prob:unique}.
%
In dimension $1$, the answer to Problem~\ref{prob:unique} is is easily seen to
be positive. The same can be shown in dimension $2$, as follows.
(Gallai in fact obtained a slightly stronger result, settling
Problem~\ref{prob:gallai} in the affirmative in dimension $2$;
see~\cite[Problem~9.14]{JT95}.)

\begin{theorem}\label{thm:2d}
  For every triangulation $K$ of $\Delta^2$ and every facet $\sigma \in K$,
  there exists a Sperner labelling of $K$ where $\sigma$ is the unique rainbow
  facet.
\end{theorem}

\begin{proof}
  Let $K$ be a triangulation of the $2$-simplex $\Delta^2$; add edges and faces
  as shown in Figure~\ref{fig:2d} to make it into a triangulation $K'$ of
  $S^2$. Let $\tau$ be the $2$-simplex corresponding to the outer face, and
  assume that $\sigma \in K$ is a given facet. Every planar triangulation is
  $3$-connected, so by Steinitz's theorem~\cite[p.~63]{MT01}, it is the graph
  of a convex polyhedron, and the dual graph is also $3$-connected. Hence, by
  Menger's theorem~\cite[p.~14]{MT01}, there exist three internally
  vertex-disjoint paths between $\sigma$ and $\tau$ in the dual graph. The
  union of these paths (in the given embedding) has three faces; let us assign
  the vertices of $K'$ embedded in the $i$-th face ($1\leq i\leq 3$) label $i$.
  It is not hard to see that this determines a Sperner labelling of $K$.
  Furthermore, if a $3$-cycle of $K'$ contains vertices of all three labels and
  differs from the face boundaries of $\sigma$ and $\tau$, then it separates a
  vertex of $\sigma$ from a vertex of $\tau$, so it is not a
  face boundary. Thus, $\sigma$ is the unique rainbow facet in the Sperner
  labelling of $K$, as claimed.
\end{proof}

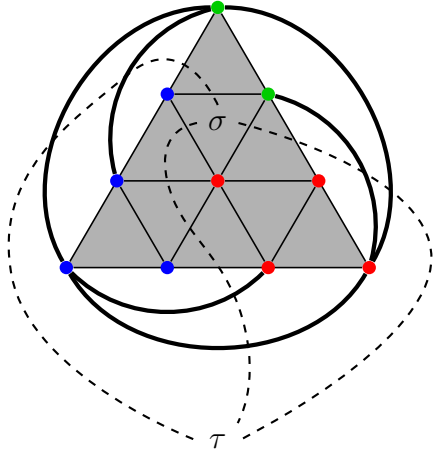
\begin{figure}
  \centering
  \begin{tikzpicture}[scale=2.3,every edge/.append style={semithick}]
    \foreach \i in {1,2,3}
    {
      \coordinate (x\i) at (120*\i-30:1);
    }
    \filldraw[draw=none,fill=black!30] (x1)--(x2)--(x3)--cycle;
    \node[vertex,fill=green!80!black] (v1) at (x1) {};
    \node[vertex,fill=blue] (v2) at (x2) {};
    \node[vertex,fill=red] (v3) at (x3) {};
    \draw (v1) edge (v2)
      (v2) edge (v3)
      (v3) edge (v1);
    \node[vertex,fill=red] (v0) at (0,0) {};
    \node[vertex,fill=blue] (v112) at (barycentric cs:v1=2/3,v2=1/3) {};
    \node[vertex,fill=blue] (v122) at (barycentric cs:v1=1/3,v2=2/3) {};
    \node[vertex,fill=green!80!black] (v113) at (barycentric cs:v1=2/3,v3=1/3)
      {};
    \node[vertex,fill=red] (v133) at (barycentric cs:v1=1/3,v3=2/3) {};
    \node[vertex,fill=blue] (v223) at (barycentric cs:v2=2/3,v3=1/3) {};
    \node[vertex,fill=red] (v233) at (barycentric cs:v2=1/3,v3=2/3) {};
    \draw (v112) edge (v113)
      (v133) edge (v233)
      (v122) edge (v223)
      (v0) edge (v112)
      (v0) edge (v122)
      (v0) edge (v223)
      (v0) edge (v233)
      (v0) edge (v113)
      (v0) edge (v133);
    \draw (v1) edge[ultra thick,bend right=60] (v2)
      (v1) edge[ultra thick,bend right=45] (v122)
      (v2) edge[ultra thick,bend right=60] (v3)
      (v2) edge[ultra thick,bend right=45] (v233)
      (v3) edge[ultra thick,bend right=60] (v1)
      (v3) edge[ultra thick,bend right=45] (v113);
    \node[minimum size=4pt,inner sep=4pt] (s) at (barycentric
      cs:v112=1/3,v113=1/3,v0=1/3) {$\sigma$};
    \node[minimum size=4pt,inner sep=4pt] (t) at (0,-1.5) {$\tau$};
    \draw[dashed,thick] plot [smooth,tension=1] coordinates {(s.west)
      (barycentric cs:v122=1/2,v0=1/2) (280:0.8) (t.north east)};
    \draw[dashed,thick] plot [smooth,tension=1] coordinates {(s.north)
      (130:0.8) (205:1.3) (t.west)};
    \draw[dashed,thick] plot [smooth,tension=1] coordinates {(s.east) (340:1.3)
      (t.east)};
  \end{tikzpicture}
  \caption{A triangulation $K$ of $\Delta^2$ (in gray) extended to a
    triangulation $K'$ of $S^2$ (thick edges and white faces; the outer face is
    labelled $\tau$). Three internally vertex-disjoint paths (dashed) in the dual
    graph partition the vertices into three subsets (red, green and blue).}
  \label{fig:2d}
\end{figure}

In this paper, we will show that the answer to Problem~\ref{prob:unique} (and
consequently, to Problem~\ref{prob:gallai}) is
negative in dimension $3$ and higher. Namely, we will prove the following
theorem in Section~\ref{sec:main}.

\begin{theorem}\label{thm:main}
  For every integer $d \geq 3$, there exist infinitely many triangulations $K$
  of $\Delta^d$ such that there exists a facet $\sigma \in K$ with the property
  that every Sperner labelling of $K$ has a rainbow facet distinct from
  $\sigma$.
\end{theorem}

\section{Terminology and preliminary results}
\label{sec:terminology}

For topological and graph-theoretic concepts not defined here, the reader is
referred to Munkres~\cite{Mun84} and to Bondy and Murty~\cite{BM08},
respectively. An introduction to Sperner's lemma---and combinatorial topology
in general---can be found in de Longueville~\cite{dL13}. The book by Mohar
and Thomassen~\cite{MT01} provides an introduction to topological graph theory.
For polytope theory, we refer the reader to Gr\" unbaum~\cite{Gru03} and to
Ziegler~\cite{Zie95}.

A $k$-colouring of a graph $G$ is a function $c:V(G) \to \{1, \ldots, k\}$ such
that $c(u) \neq c(v)$ whenever $uv \in E(G)$. The minimum $k$ for which a
$k$-colouring of $G$ exists is known as the \emph{chromatic number} of $G$,
denoted by $\chi(G)$. If $\chi(G) \leq k$ or $\chi(G)=k$, then $G$ is said to
be \emph{$k$-colourable} or \emph{$k$-chromatic}, respectively. Furthermore, a
graph $G$ is said to be \emph{$k$-critical} (resp.\ \emph{$k$-vertex-critical})
if $G$ is $k$-chromatic and deleting any edge (resp.\ vertex) results in a
$(k-1)$-colourable graph. A graph is said to be \emph{$k$-vertex-connected} if
it has at least $k+1$ vertices, and removing any set of $k-1$ vertices does not
disconnect the graph.

A (geometric) \emph{simplicial complex} $K$ is a collection of geometric
simplices of various dimensions satisfying two conditions: firstly, if $\sigma
\in K$ and $\tau$ is a face of $\sigma$, then $\tau \in K$; secondly, if
$\sigma, \tau \in K$ and $\sigma \cap \tau \neq \emptyset$, then $\sigma \cap
\tau$ is a face of both $\sigma$ and $\tau$. We denote the regular
$d$-dimensional simplex, unit $d$-ball and unit $d$-sphere by $\Delta^d$, $B^d$
and $S^d$, respectively. Simplices of $K$ that are not strictly contained in
other simplices are the \emph{facets} of $K$. The union of all simplices of
$K$, known as the \emph{polyhedron} of $K$, is denoted by $|K|$. The
\emph{$k$-skeleton} of a simplicial complex $K$, denoted $K^{(k)}$, consists of
all the simplices of $K$ of dimension at most $k$; note that $K^{(1)}$ consists
of vertices and edges, i.e., it is a graph. If two topological spaces $X$ and
$Y$ are homeomorphic, we write $X \cong Y$. A simplicial complex $K$ such that
$X \cong |K|$, if one exists, is called a \emph{triangulation} of $X$. The
\emph{join} of two simplicial complexes $K$ and $L$ is denoted by $K \star L$.

A \emph{quadrangulation} of a surface (i.e., a $2$-dimensional manifold) $X$ is
a graph embedded in $X$ so that every face of the embedding is bounded by a
$4$-cycle. In~\cite{KS15}, the definition of quadrangulations was extended to
triangulable topological spaces in the following way. A graph $G$ is a
\emph{quadrangulation} of a topological space $X$ if $G$ is a spanning subgraph
of the $1$-skeleton $K^{(1)}$ of some triangulation $K$ of $X$, with the
property that the vertices of every facet of $K$ induce a complete bipartite
subgraph of $G$ with at least one edge. This definition coincides with the
usual definition in the case when $X$ is a $2$-dimensional manifold.
In~\cite[Lemma~3.2]{KS15}, the following characterisation was given of
non-bipartite quadrangulations of the $d$-dimensional real projective space
$\RP^d$:

\begin{lemma}\label{lem:quad}
  A graph $G$ is a non-bipartite quadrangulation of $\RP^d$ if and only if
  there exists a centrally symmetric triangulation $K$ of $S^d$ and a
  $2$-colouring of the vertices with colours black and white, such that no pair
  of antipodal vertices is adjacent, antipodal vertices receive different
  colours, every facet of $K$ contains a black and a white vertex, and $G$ is
  obtained from the $1$-skeleton $K^{(1)}$ by first deleting all monochromatic
  edges, and then identifying antipodal vertices and edges.
\end{lemma}

If $P$ is a (convex) $d$-polytope, its $(d-1)$-dimensional faces are its
\emph{facets}. If every facet is a simplex, then the \emph{boundary} of $P$,
denoted by $\partial P$, is a simplicial complex such that $|\partial P| \cong
S^{d-1}$, i.e., $\partial P$ is a triangulation of $S^{d-1}$.  The convex hull
of the vectors of the standard orthonormal basis and their negatives in $\RR^d$
is known as the the $d$-dimensional \emph{cross polytope}.

\section{Relation to chromatic number}
\label{sec:chromatic}

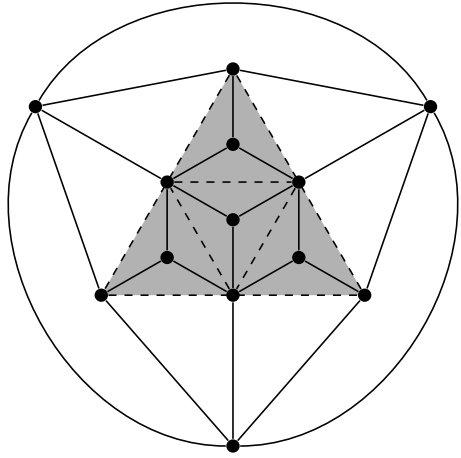
\begin{figure}
  \centering
  \begin{tikzpicture}[scale=2,every edge/.append style={semithick}]
    \foreach \i in {1,2,3}
    {
      \coordinate (x\i) at (120*\i-30:1);
    }
    \filldraw[draw=none,fill=black!30] (x1)--(x2)--(x3)--cycle;
    \foreach \i in {1,2,3}
    {
      \node[vertex] (v\i) at (x\i) {};
    }
    \node[vertex] (v12) at (barycentric cs:v1=1/2,v2=1/2) {};
    \node[vertex] (v13) at (barycentric cs:v1=1/2,v3=1/2) {};
    \node[vertex] (v23) at (barycentric cs:v2=1/2,v3=1/2) {};
    \node[vertex] (u0) at (0,0) {};
    \node[vertex] (u1) at (barycentric cs:v1=1/3,v12=1/3,v13=1/3) {};
    \node[vertex] (u2) at (barycentric cs:v2=1/3,v12=1/3,v23=1/3) {};
    \node[vertex] (u3) at (barycentric cs:v3=1/3,v13=1/3,v23=1/3) {};
    \foreach \i in {1,2,3}
    {
      \node[vertex] (w\i) at (120*\i+150:1.5) {};
    }
    \draw[dashed] (v1) edge (v2)
      (v2) edge (v3)
      (v3) edge (v1);
    \draw[dashed] (v12) edge (v13)
      (v13) edge (v23)
      (v23) edge (v12);
    \draw (w1) edge[bend right=60] (w2) 
      (w2) edge[bend right=60] (w3)
      (w3) edge[bend right=60] (w1);
    \draw (w1) edge (v2)
      (w1) edge (v3)
      (w1) edge (v23);    
    \draw (w2) edge (v1)
      (w2) edge (v3)
      (w2) edge (v13);
    \draw (w3) edge (v1)
      (w3) edge (v2)
      (w3) edge (v12);
    \draw (u0) edge (v12)
      (u0) edge (v13)
      (u0) edge (v23);
    \draw (u1) edge (v1)
      (u1) edge (v12)
      (u1) edge (v13);
    \draw (u2) edge (v2)
      (u2) edge (v12)
      (u2) edge (v23);
    \draw (u3) edge (v3)
      (u3) edge (v13)
      (u3) edge (v23);
  \end{tikzpicture}
  \caption{A triangulation $K$ (dashed edges) of a $2$-simplex (in gray), and
    the graph $G_K$ (solid edges).}
  \label{fig:gallai}
\end{figure}

Gallai used Sperner's lemma to construct a family of $4$-chromatic planar
graphs; see~\cite[Problem~9.14]{JT95} or~\cite{NT75}. Let $K$ be a
triangulation of $\Delta^d$, and define the graph $G_K$ as follows (see
Figure~\ref{fig:gallai} for an illustration). There are three types of vertices
of $G_K$:
\begin{itemize}
  \item[(V1)] The vertices $u_1, \ldots, u_n$ of $K$;
  \item[(V2)] A vertex $v_{\rho}$ for every $d$-simplex $\rho \in K$;
  \item[(V3)] A vertex $w_{\sigma}$ for every $(d-1)$-simplex $\sigma \in
    \Delta^d$.
\end{itemize}
The edge set of $G_K$ is defined as $E(G_K)=E_1 \cup E_2 \cup E_3$, where:
\begin{itemize}
  \item[(E1)] $u_iv_{\rho} \in E_1$ if $u_i \in \rho$, where
    $u_i$ is a vertex of $K$ and $\rho$ is a $d$-simplex of $K$;
  \item[(E2)] $u_iw_{\sigma} \in E_2$ if $u_i$ belongs to the
    facet $\sigma$ of $\Delta^d$;
  \item[(E3)] $w_{\sigma}w_{\tau} \in E_3$ for any pair of distinct $(d-1)$-simplices
    $\sigma, \tau \in \Delta^d$.
\end{itemize}

Note that the vertices of type (V1) and (V2) form independent sets in
$G_K$, the vertices of type (V3) induce a $(d+1)$-clique in $G_K$, and there are no
edges between the vertices of type (V2) and of type (V3). It follows that
$G_K$ is $(d+2)$-colourable---for example, by colouring the clique induced
by (V3) using colours $1,\ldots,d+1$, the vertices of type (V2) by colour $1$,
and the vertices of type (V1) by colour $d+2$.

What is more interesting is that $G_K$ is not $(d+1)$-colourable (and therefore
$G_K$ is $(d+2)$-chromatic). Indeed,
Gallai observed that the statement that $G_K$ is not $(d+1)$-colourable is
equivalent to Sperner's lemma; we formalise it as the following proposition.

\begin{proposition}\label{prop:equiv}
  $G_K$ is not $(d+1)$-colourable if and only if every Sperner labelling of $K$
  has a rainbow facet.
\end{proposition}

\begin{proof}
  Suppose for a contradiction that there exists a Sperner labelling $\lambda$
  of $K$ with no rainbow facet. Then we can define a $(d+1)$-colouring $c$ of
  $G_K$ as follows. For each vertex $u_i$ of type (V1), let
  $c(u_i)=\lambda(u_i)$. For each vertex $w_{\sigma}$ of type (V3), let $u_j$
  be the unique vertex of $\Delta^d$ not belonging to $\sigma$, and set
  $c(w_{\sigma})=\lambda(u_j)$. Finally, each vertex $v_{\rho}$ of type (V2) is
  adjacent to the $d+1$ vertices of a facet $\rho \in K$, which is not rainbow
  by assumption. Hence, the colouring $c$ can be extended to a
  $(d+1)$-colouring of $G_K$.

  Conversely, suppose that there is a $(d+1)$-colouring of $G_K$. The vertices
  of $G_K$ of type (V3) ensure that the labelling of $K$ induced by the
  colouring is a Sperner labelling, and there is no rainbow facet in $K$ since
  the vertices of every facet in $K$ are adjacent to a vertex of type (V2) in
  $G_K$.
\end{proof}

Observe that if $K$ is a triangulation of the $2$-dimensional simplex, then
$G_K$ is a $4$-chromatic planar graph with exactly $4$ triangles. This is best
possible in view of the sharpening of Grötzsch's theorem due to
Grünbaum~\cite{Gru63} and Aksenov~\cite{Aks74}, which states that every planar
graph with at most $3$ triangles is $3$-colourable. Indeed, Borodin et
al.~\cite{BDKLY14} characterised the $4$-chromatic planar graphs with exactly
$4$ triangles; in this characterisation, the graphs $G_K$ belong among those
graphs that can be obtained from the complete graph on $4$ vertices by
replacing a single vertex of degree $3$ by a `critical' patch (in a sense made
precise in~\cite{BDKLY14}).

It is easy to see that for any triangulation $K$ of the $1$-simplex, the graph
$G_K$ is an odd cycle and thus $3$-critical. Gallai~\cite{JT95,NT75} proved
that $G_K$ is $4$-critical whenever $K$ is a triangulation of the $2$-simplex;
we note that this also follows from Theorem~\ref{thm:projective} in conjunction
with~\cite[Theorem~5.4]{GT97} (see the discussion at the end of
Section~\ref{sec:quadrangulation}).

In 1969 Gallai asked the following question, presented as Problem~9.14 by
Jensen and Toft~\cite{JT95}. The question is also mentioned in~\cite{NT75,RT85}.

\begin{problem}\label{prob:gallai}
  If $d \geq 3$, and $K$ is a triangulation of $\Delta^d$, is the graph $G_K$
  $(d+2)$-critical?
\end{problem}

The following observation shows that a positive answer to Problem~\ref{prob:gallai}
would imply a positive answer to Problem~\ref{prob:unique}.

\begin{observation}\label{obs:critical}
  Let $K$ be a triangulation of $\Delta^d$. If $G_K$ is $(d+2)$-critical,
  then for any facet $\sigma \in K$, there exists a Sperner labelling of
  $K$ where $\sigma$ is the unique rainbow facet.
\end{observation}

\begin{proof}
  Suppose $G_K$ is $(d+2)$-critical, for some triangulation $K$ of $\Delta^d$.
  Then for any facet $\sigma \in K$, the graph $G_K-v_{\sigma}$ is
  $(d+1)$-colourable. Hence, the labelling of $K$ induced by the
  $(d+1)$-colouring of $G-v_{\sigma}$ is a Sperner labelling with no rainbow
  facet distinct from $\sigma$.
\end{proof}

It is easy to see that Observation~\ref{obs:critical}, in conjunction with
Theorem~\ref{thm:main}, immediately settles Gallai's question in the negative.

\begin{theorem}\label{thm:gallai-counterexample}
  For every $d \geq 3$, there exist infinitely many triangulations $K$ of
  $\Delta^d$ such that $G_K$ is not $(d+2)$-critical.
\end{theorem}

\section{$G_K$ as projective quadrangulations}
\label{sec:quadrangulation}

Jensen and Toft~\cite[Problem~9.14]{JT95} also raised the question of whether
the graphs $G_K$ belong to a `larger class of $(d+2)$-chromatic graphs defined
in purely graph-theoretic terms'. While we cannot do away with the topology, we
can show that the graphs $G_K$ belong to a larger class of graphs whose
chromatic number is large for `topological reasons', namely the class of
quadrangulations of projective spaces. Indeed, a result by the first two
authors~\cite[Theorem~1.1]{KS15}, generalising an earlier result of
Youngs~\cite{You96}, shows that all non-bipartite quadrangulations of the
$d$-dimensional projective space $\mathbb R\textup P^d$ are at least
$(d+2)$-chromatic.  

\begin{theorem}\label{thm:projective}
  If $K$ is a triangulation of the $d$-simplex, then $G_K$ is a non-bipartite
  quadrangulation of the $d$-dimensional projective space $\RP^d$.
\end{theorem}

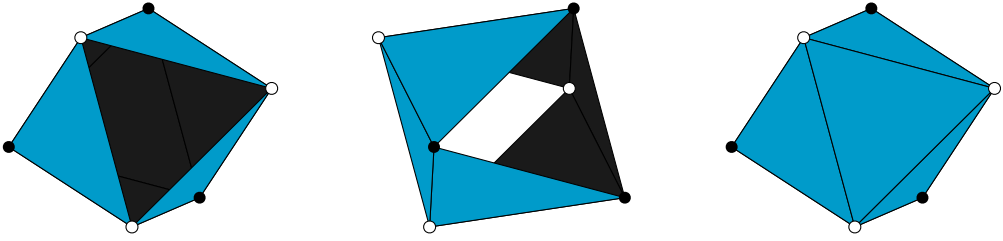
\begin{figure}
  \centering
  \makeatletter
  \define@key{cylindricalkeys}{angle}{\def\myangle{#1}}
  \define@key{cylindricalkeys}{radius}{\def\myradius{#1}}
  \define@key{cylindricalkeys}{z}{\def\myz{#1}}
  \tikzdeclarecoordinatesystem{cylindrical}%
  {%
    \setkeys{cylindricalkeys}{#1}%
    \pgfpointadd{\pgfpointxyz{0}{0}{\myz}}{\pgfpointpolarxy{\myangle}{\myradius}}
  }
  \begin{tikzpicture}[scale=1.5,rotate=-45,z=10,line join=bevel]
  \def\distance{6}
  \foreach \i in {1,2,3}
  {
    \coordinate (u\i) at (cylindrical
      cs:angle=120*\i-120,radius=1,z=-\distance-0.75pt);
    \coordinate (u\i') at (cylindrical
      cs:angle=120*\i+60,radius=1,z=-\distance);
    \coordinate (v\i) at (cylindrical cs:angle=120*\i-120,radius=1,z=0.75);
    \coordinate (v\i') at (cylindrical cs:angle=120*\i+60,radius=1,z=-0.75);
    \coordinate (w\i) at (cylindrical
      cs:angle=120*\i+60,radius=1,z=\distance+0.75);
    \coordinate (w\i') at (cylindrical
      cs:angle=120*\i-120,radius=1,z=\distance);
  }
  \draw[fill opacity=0.9,fill=gray!20!black] (u2) -- (u3) -- (u1) -- cycle;
  \draw[fill opacity=0.9,fill=gray!20!black] (u2) -- (u3) -- (u1') -- cycle;
  \draw[fill opacity=0.9,fill=gray!20!black] (u2') -- (u3) -- (u1) -- cycle;
  \draw[fill opacity=0.9,fill=gray!20!black] (u2) -- (u3') -- (u1) -- cycle;
  \draw[fill opacity=0.9,fill=cyan!80!black] (u2') -- (u3) -- (u1') -- cycle;
  \draw[fill opacity=0.9,fill=cyan!80!black] (u2) -- (u3') -- (u1') -- cycle;
  \draw[fill opacity=0.9,fill=cyan!80!black] (u2') -- (u3') -- (u1) -- cycle;
  \draw[fill opacity=0.9,fill=gray!20!black] (v1') -- (v2) -- (v3') -- cycle;
  \draw[fill opacity=0.9,fill=gray!20!black] (v1) -- (v2) -- (v3') -- cycle;
  \draw[fill opacity=0.9,fill=gray!20!black] (v1) -- (v2') -- (v3') -- cycle;
  \draw[fill opacity=0.9,fill=cyan!80!black] (v1) -- (v2') -- (v3) -- cycle;
  \draw[fill opacity=0.9,fill=cyan!80!black] (v1') -- (v2') -- (v3) -- cycle;
  \draw[fill opacity=0.9,fill=cyan!80!black] (v1') -- (v2) -- (v3) -- cycle;
  \draw[fill opacity=0.9,fill=gray!20!black] (w1') -- (w2) -- (w3') -- cycle;
  \draw[fill opacity=0.9,fill=gray!20!black] (w1) -- (w2') -- (w3') -- cycle;
  \draw[fill opacity=0.9,fill=gray!20!black] (w1') -- (w2') -- (w3) -- cycle;
  \draw[fill opacity=0.9,fill=cyan!80!black] (w1) -- (w2) -- (w3') -- cycle;
  \draw[fill opacity=0.9,fill=cyan!80!black] (w1) -- (w2') -- (w3) -- cycle;
  \draw[fill opacity=0.9,fill=cyan!80!black] (w1') -- (w2) -- (w3) -- cycle;
  \draw[fill opacity=0.9,fill=cyan!80!black] (w1) -- (w2) -- (w3) -- cycle;
  \foreach \i in {1,...,3}
  {
    \fill (u\i) circle (1.5pt);
    \fill (v\i) circle (1.5pt);
    \fill (w\i') circle (1.5pt);
    \filldraw[color=black,fill=white] (u\i') circle (1.5pt);
    \filldraw[color=black,fill=white] (v\i') circle (1.5pt);
    \filldraw[color=black,fill=white] (w\i) circle (1.5pt);
  }
  \end{tikzpicture}
  \caption{The complexes $P-\alpha_w$ (left), $P-\alpha_w-\alpha_b$ (centre) and
    $P-\alpha_b$ (right) in the proof of Theorem~\ref{thm:projective} (case $d=2$).}
  \label{fig:cross}
\end{figure}

\begin{proof}
  Let $P$ be the $2$-coloured complex obtained from the boundary
  complex $\partial Q^{d+1}$ of the $(d+1)$-dimensional cross polytope
  $Q^{d+1}$ by colouring all the vertices of one facet (say
  $\alpha_w$) white, and all the vertices of the complementary facet
  (say $\alpha_b$) black. The complex $P-\alpha_w-\alpha_b$ obtained
  by deleting the facets $\alpha_w$ and $\alpha_b$ from $P$ contains
  $\partial \alpha_w$ and $\partial \alpha_b$ as
  subcomplexes. Similarly, $P-\alpha_w$ and $P-\alpha_b$ contain
  $\partial\alpha_w$ and $\partial\alpha_b$, respectively (see
  Figure~\ref{fig:cross} for an illustration). Glue
  $P-\alpha_w-\alpha_b$ to $P-\alpha_w$ along $\partial \alpha_w$ (by
  identifying the corresponding vertices of $\partial\alpha_w$), and
  then glue the result to $P-\alpha_b$ along $\partial \alpha_b$. The
  resulting complex $C$ is a centrally symmetric triangulation of
  $S^d$. Note that $C$ has a unique facet $\alpha'_w$ with all
  vertices white (coming from the facet $\alpha_w$ of $P-\alpha_b$)
  and a unique facet $\alpha'_b$ with all vertices black.

  Since $K$ is a subdivision of a $d$-simplex $\Delta^d$, there is a
  sequence of stellar moves (see e.g.~\cite[Theorem~4.5]{Lic99}) that
  transforms $\Delta^d$ into $K$. Assign to each vertex $v$ of the
  all-white facet $\alpha'_w$ of $C$ a vertex of $\Delta^d$ using an
  arbitrary bijection, and assign the same vertex of $\Delta^d$ to the
  antipode of $v$. Perform the corresponding sequence of stellar moves
  on $\alpha'_w$, colouring all new vertices white. Proceed
  analogously for $\alpha'_b$ with colours reversed. Finally, perform
  a stellar subdivision of each all-white facet with a black vertex,
  and of each all-black facet with a white vertex. Call the resulting
  $2$-coloured complex $\tilde K$.

  Since $\tilde K$ is a centrally symmetric subdivision of the centrally
  symmetric triangulation $C$ of $S^d$, it follows that $\tilde K$ is a
  centrally symmetric triangulation of $S^d$. Moreover, $\tilde K$ is equipped
  with an
  antisymmetric $2$-colouring such that no facet is
  monochromatic. Therefore, by Lemma~\ref{lem:quad}, the graph $G$
  obtained from the $1$-skeleton of $\tilde K$ by deleting
  monochromatic edges, and by identifying antipodal vertices, is a
  non-bipartite quadrangulation of $\RP^d$. We claim that
  $G \cong G_K$.  To see this, note that $\tilde K$ contains three
  types of vertices, each appearing in antipodal pairs consisting of a
  white vertex and a black vertex:
  \begin{itemize}
  \item[(W1)] vertices in the black and white copies of $K$ not
    contained in $P-\alpha_w-\alpha_b$;
    \item[(W2)] vertices subdividing each facet of the black or white copy of $K$;
    \item[(W3)] vertices in $P-\alpha_w-\alpha_b$.
  \end{itemize}
  There is an obvious bijection between the vertices of $G$ and those of $G_K$:
  the vertices of $G$ of type (W1), (W2) and (W3) are in bijection with the
  vertices of $G_K$ of type (V1), (V2) and (V3), respectively. It is not hard to
  check that the edges of $G$ are in bijection with the edges of $G_K$.
\end{proof}

Non-bipartite quadrangulations of $\RP^1$ are odd cycles, so they are
$3$-critical. Youngs~\cite{You96} showed that non-bipartite quadrangulations of
$\RP^2$ are $4$-chromatic, and conjectured that they are $4$-critical as long
as they contain no non-bipartite quadrangulation of $\RP^2$ as a proper
subgraph; this was subsequently proved by Gimbel and Thomassen~\cite{GT97}. For
$d \geq 3$, non-bipartite quadrangulations of $\RP^d$ are not $(d+2)$-critical
in general; indeed, their chromatic number is unbounded~\cite{KS15}.

\section{Proof of the main result}
\label{sec:main}

Our construction uses a specific convex $4$-polytope known as $H_8$ and
described in \cite[Section 7.2]{Gru03}. This polytope is $2$-neighbourly
(recall that a polytope is \emph{$t$-neighbourly} if every set of $t$ vertices
constitutes a face). Motzkin~\cite{Mot57} claimed that for even $d$, every
$d/2$-neighbourly $d$-polytope is combinatorially equivalent to a cyclic
polytope (see~\cite{Gru03} for the necessary definitions). This assertion is
true for $d$-polytopes with at most $d+3$ vertices~\cite{Gal63} but fails in
general; $H_8$ is one of the two counterexamples for $d=4$ constructed
in~\cite{Gru03}. Our example is in line with Sturmfels' remark~\cite{Stu88}
that $d+4$ vertices seem to be a `threshold of counterexamples'.

The convex $4$-polytope $H_8$ has $8$ vertices $A, B, C, D, E, F, G, Z$ and the
following are the $20$ facets of its boundary complex.
\begin{center}
  \begin{tabular}{ccccc}
    $ABCD$ & $\underline{ABCG}$ & $\underline{ABDE}$ & $ABEF$ & $\underline{ABFZ}$\\
    \boldmath$\underline{ABGZ}$ & $ACDZ$ & $\underline{ACGZ}$ & $\underline{ADEZ}$ & $AEFZ$\\
    $BCDE$ & $BCEF$ & $BCFG$ & $BFGZ$ & \boldmath$\underline{CDEF}$\\
    $\underline{CDFG}$ & $CDGZ$ & $\underline{DEFG}$ & $DEGZ$ & $EFGZ$
  \end{tabular}
\end{center}
We underlined the facets relevant for the proof of Lemma~\ref{lem:h8} below.
The two facets in bold play a key role in the proof of Theorem~\ref{thm:main},
and it will be convenient to name them. Throughout the remainder of this paper,
we let $\sigma_0 := ABGZ$ and $\sigma_1 := CDEF$.

\begin{lemma}\label{lem:h8}
  Let the vertices of $H_8$ be labelled by $1, 2, 3, 4$ in such a way that the
  facets $\sigma_0$ and $\sigma_1$ of $\partial H_8$ are rainbow. Then there exists
  another rainbow facet of $\partial H_8$.
\end{lemma}

\begin{proof}
  We proceed by contradiction. Let $\varphi$ be the given labelling. Suppose
  that $\sigma_0$ and $\sigma_1$ are the only rainbow facets of $\partial H_8$,
  and assume (without loss of generality) that the vertices $A, B, G, Z$ are
  labelled $1, 2, 3, 4$ in order. Observe the following direct consequence of
  the assumptions.

  \begin{claim}\label{cl:diff}
    If a facet $\tau$ of $\partial H_8$ shares a $2$-dimensional face with
    $\sigma_i$ (for some $i\in\Setx{0,1}$), then the two vertices in the
    symmetric difference of $\tau$ and $\sigma_i$ have different labels.
  \end{claim}
  
  Applying Claim~\ref{cl:diff} to the facet $ABCG$ (and $\sigma_0$) we find
  $\varphi(C) \neq \varphi(Z) = 4$. Similarly, $\varphi(C)\neq 2$ (considering
  $ACGZ$ and $\sigma_0$) and $\varphi(C)\neq 3$ (considering $DEFG$ and
  $\sigma_1$). Consequently, $\varphi(C) = 1$.

  Neither $E$ nor $F$ have label $3$, as seen by applying Claim~\ref{cl:diff}
  to the facets $CDFG$ and $ABFZ$, respectively. Since $\sigma_1$ is rainbow,
  these vertices do not have label $1$ either, since it is already taken by
  $C$. Thus, $\Setx{\varphi(E),\varphi(F)} = \Setx{2,4}$. This implies that
  $\varphi(D) = 3$.

  The facet $ABDE$ shares two vertices $A,B$ with $\sigma_0$. Since $ABDE$ is
  not rainbow, we deduce that $\Setx{\varphi(D),\varphi(E)} \neq
  \Setx{\varphi(G),\varphi(Z)} = \Setx{3,4}$. Hence, $\varphi(E)\neq 4$ as we
  know that $\varphi(D)=3$. This leaves only the option $\varphi(E) = 2$ for
  the label of $E$. But then the facet $ADEZ$ is rainbow, a contradiction.
  \end{proof}

We remark that $H_8$ is far from the only polytope satisfying an analogue of
Lemma~\ref{lem:h8}. While it is the only such polytope among $4$-polytopes on
$8$ vertices which are \emph{simplicial} (all their faces are simplices; these
polytopes are listed in~\cite{GS67}), the situation is much different when the
number of vertices is increased. For instance, a complete search of the $23$
$2$-neighbourly $4$-polytopes on $9$ vertices listed in~\cite{AS73} showed that
$18$ of them have the same property as $H_8$.

\begin{proof}[Proof of Theorem~\ref{thm:main}]
  Let $K$ be the simplicial complex obtained from the boundary complex of $H_8$
  by deleting the simplex $\sigma_0$. Since $H_8$ is a $4$-polytope, we have
  $H_8 \cong B^4$, so $\partial H_8 \cong S^3$, and therefore $K \cong B^3$,
  or, equivalently, $K$ a triangulation of the $3$-simplex with vertices
  $A,B,G,Z$.

  For an integer $d \geq 3$, let $\rho$ be a $(d-4)$-simplex, and define $K_d=K
  \star \rho$ (we take the convention that the empty simplex has dimension
  $-1$; in particular, $K_3=K$). We can view $K_d$ as a triangulation of the
  $d$-simplex with vertices $\{A,B,G,Z\}\cup V(\rho)$. We will show that in any
  Sperner labelling of $K_d$, there is a rainbow facet distinct from
  $\sigma:=\sigma_1 \star \rho$. Fix any Sperner labelling $\lambda:V(K_d) \to \{1, \ldots,
  d+1\}$ of $K_d$ (observe that $\rho$ is rainbow with respect to $\lambda$).
  If $\sigma$ is not rainbow, then by Sperner's lemma, there must be a rainbow
  facet distinct from $\sigma$ and we are done. We may thus assume that
  $\sigma$ is rainbow. This means that the simplices $\sigma_0$ and $\sigma_1$
  receive the same set of labels. In particular, the labelling $\lambda'$ of
  $V(K)$ induced by $\lambda$ uses $4$ labels. By Lemma~\ref{lem:h8}, some
  $3$-simplex $\tau \in K$ distinct from $\sigma_1$ is rainbow with respect to
  $\lambda'$. Hence, $\tau \star \rho$ is a rainbow $d$-simplex with respect to
  $\lambda$. We have thus shown that any Sperner labelling of $K_d$ has a
  rainbow $d$-simplex distinct from $\sigma$. This provides a negative answer 
  to Gallai's question for all $d \geq 3$.
  
  It remains to show that, for every $d \geq 3$, there are infinitely many
  triangulations of $\Delta^d$ with this property. Consider any facet $\pi \in K_d$ distinct from $\sigma$, and let $L$
  be a triangulation of $\pi$ such that $\partial L=\partial \pi$. Let $K_d'$
  be the complex obtained from $K_d$ by replacing $\pi$ by $L$. Fix any Sperner
  labelling $\lambda:V(K_d') \to \{1, \ldots, d+1\}$, and suppose that no facet
  of $K_d'$ distinct from $\sigma$ is rainbow. The $d+1$ vertices in $\partial
  L$ must receive labels $1, \ldots, d+1$. Therefore, the labelling induced by
  $L$ is a Sperner labelling of $L$, so by Sperner's lemma, there is at least
  one rainbow facet in $L$. This shows that any Sperner labelling of $K_d'$ has
  a rainbow $d$-simplex distinct from $\sigma$.

  By iterating this construction, we can obtain infinitely many triangulations
  $K$ of $\Delta^d$ with the property that every Sperner labelling of $K$ has a
  rainbow $d$-simplex distinct from $\sigma$.
\end{proof}

It may well be that even in higher dimensions, the property in
Problem~\ref{prob:unique} holds for triangulations of a special type. For
example, one could take cyclic polytopes (a simple and well-understood class)
and consider the triangulations associated with them in the following sense.
Given a polytope $P$, let us say that a triangulation is \emph{associated with
$P$} if it is obtained from the boundary of $P$ by removing a facet $\tau$ and
projecting the rest of the boundary into the convex hull of the vertex set of
$\tau$.

\begin{conjecture}
  Problem~\ref{prob:unique} has an affirmative answer if $K$ is a triangulation
  of $\Delta^d$ associated with a $d$-dimensional cyclic polytope.
\end{conjecture}

\section*{Acknowledgements}

We are indebted to Frank Lutz for the valuable information on small
triangulations of manifolds collected in~\cite{Lutz}. We used Sage to test
these triangulations (as well as those obtained from~\cite{AS73,GS67}) for
specific properties, and we thank the developers of this fine piece of
software. Finally, we would like to thank the referees for making some useful
suggestions to improve the readability of the paper.

\end{document}